\documentclass[conference]{IEEEtran}
\IEEEoverridecommandlockouts
\usepackage{braket}

\usepackage{amsthm, thmtools}

\usepackage{bigints} 	
\usepackage{steinmetz}  
\usepackage{amsmath,amssymb,amsfonts} 
\usepackage{graphics}  
\usepackage{graphicx}
\usepackage{setspace}
\usepackage{upgreek}
\usepackage{mathtools}
\usepackage{fancyhdr}
\usepackage{float}
\usepackage{tocloft}
\usepackage{tcolorbox}
\usepackage{comment}
\usepackage{empheq, blkarray}
\usepackage{subfig}
\usepackage{subfloat}
\usepackage[font={small,it}]{caption}
\declaretheoremstyle[%
headfont=\normalfont\bfseries,
within=section,headpunct={.\smallskip \newline}, 
bodyfont =\itshape, 
spaceabove = 8pt,spacebelow = 8pt]%
{mythmstyle}

\makeatletter
\newcommand{\removelatexerror}{\let\@latex@error\@gobble}
\makeatother
\usepackage{verbatim}
\tcbuselibrary{breakable}
\usepackage{eso-pic}
\newcommand\numberthis{\addtocounter{equation}{1}\tag{\theequation}}
\theoremstyle{definition}
\newtheorem{definition}{Definition}
\newtheorem{assumption}{Assumption}
\theoremstyle{plain}
\declaretheorem[name=Theorem]{thm}
\declaretheorem[name=Lemma]{lem}
\declaretheorem[name=Proposition]{prop}
\declaretheorem[name=Corollary]{cor}

\theoremstyle{remark}
\newtheorem{remark}{Remark}
\newtheorem{example}{Example}

\DeclareMathOperator{\dom}{dom}

\DeclareMathOperator{\diag}{diag}

\DeclareMathOperator{\interior}{int}
\DeclareMathOperator{\lse}{lse}

\DeclareMathOperator{\vecmax}{vecmax}

\DeclareMathOperator{\gra}{gra}
\DeclareMathOperator{\Jacob}{\mathbf{J}}
\usepackage{caption}
\usepackage{enumitem}

\usepackage[colorlinks=true,
linkcolor=black,
urlcolor=black,
citecolor=black]{hyperref}

\begin{document}
	\title{\LARGE \bf On the Properties of the Softmax Function with Application in Game Theory and Reinforcement Learning}
	\author{Bolin Gao and Lacra Pavel%
		\thanks{B. Gao and L. Pavel are with the Department of Electrical and Computer Engineering, University of Toronto, Toronto, ON, M5S 3G4, Canada. Emails:
			{\tt\small bolin.gao@mail.utoronto.ca, pavel@ece.utoronto.ca}}%
	}

\maketitle
\thispagestyle{plain}
\pagestyle{plain}
\IEEEpeerreviewmaketitle
	\begin{abstract}
	In this paper, we utilize results from convex analysis and monotone operator theory to derive additional properties of the softmax function that have not yet been covered in the existing literature. In particular, we show that the softmax function is the monotone gradient map of the log-sum-exp function. By exploiting this connection, we show that the inverse temperature parameter $\lambda$ determines the Lipschitz and co-coercivity properties of the softmax function. We then demonstrate the usefulness of these properties through an application in game-theoretic reinforcement learning.  
	\end{abstract}

	\section{Introduction}
	
    The softmax function is one of the most well-known functions in science and engineering and has enjoyed widespread usage in fields such as game theory \cite{Young, Sandholm, Goeree}, reinforcement learning \cite{Sutton} and machine learning \cite{Goodfellow, Bishop}.
    From a game theory and reinforcement learning perspective, the softmax function maps the raw payoff or the score (or Q-value) associated with a payoff to a mixed strategy \cite{Young, Sandholm, Sutton}, whereas from the perspective of multi-class logistic regression, the softmax function maps a vector of logits (or feature variables) to a posterior probability distribution \cite{Goodfellow, Bishop}. The broader engineering applications involving the softmax function are numerous; interesting examples can be found in the fields of VLSI and neuromorphic computing, see \cite{Zuino, Yuille, Elfadel, Genewein}. 
	
	The term ``softmax" is a portmanteau of ``soft" and  ``argmax" \cite{Goodfellow}. The function first appeared in the work of Luce \cite{Luce}, although its coinage is mostly credited to Bridle \cite{Bridle}. Depending on the context in which the softmax function appears, it also goes by the name of Boltzmann distribution \cite{Young, Sutton, Kou}, Gibbs map \cite{Coucheney_Learning, Laraki_Higher_Game}, logit map, logit choice rule, logit response function \cite{Young, Sandholm, Goeree, Hof_Hop, qre, Merti_Learning, Cominetti} or (smooth) perturbed best response function \cite{Leslie, Hopkins_Learning}. The reader should take care in distinguishing the softmax function used in this paper from the log-sum-exp function, which is often also referred to as the ``softmax" (since the log-sum-exp is a soft approximation of the vector-max function \cite{Boyd, Wets}).
	
	There are many factors contributing to the wide-spread usage of the softmax function. In the context of reinforcement learning, the softmax function ensures a trade-off between exploitation and exploration, in that every strategy in an agent's possession has a chance of being explored. Unlike some other choice mechanisms such as $\epsilon$-greedy \cite{Sutton}, the usage of softmax selection rule\footnote{In this paper, we refer to the softmax function interchangeably as the softmax operator, softmax map, softmax choice, softmax selection rule, or simply, the softmax.} is favorably supported by experimental literature in game theory and reinforcement learning as a plausible model for modeling real-life decision-making. For instance, in \cite{Kianercy}, the authors noted that the behavior of monkeys during reinforcement learning experiments is consistent with the softmax selection rule. Furthermore, the input-output behavior of the softmax function has been compared to lateral inhibition in biological neural networks \cite{Goodfellow}. For additional discussions on the connections between softmax selection rule and the neurophysiology of decision-making, see \cite{Daw, Lee1, Cohen, Bossaerts}. From the perspective of game theory, the softmax function characterizes the so-called ``logit equilibrium", which accounts for incomplete information and random perturbation of the payoff during gameplay and has been noted for having better versatility in describing the outcomes of gameplay as compared to the Nash equilibrium \cite{Goeree, qre}.

	\begin{figure}[htp!]
		\begin{center}
			\includegraphics[scale=0.8]{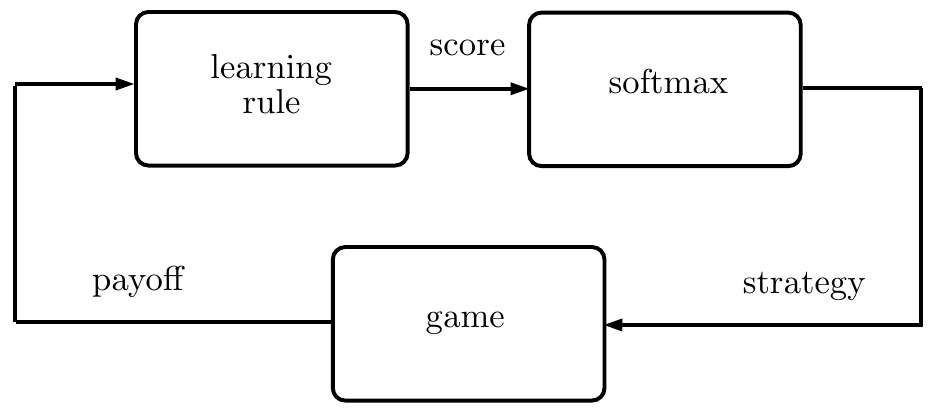}
		\end{center}
		\caption{High-level representation of a game-theoretic multi-agent reinforcement learning scheme with the softmax selection rule. In this learning scenario, the players each choose some strategy, play the game and receive real-valued payoffs. The players then use some learning rule to independently convert the payoffs into scores. Finally, each player uses the softmax to select the next strategy.} 
		\label{fig:Ctrl_RL_Sys}
	\end{figure}
	
	Despite the intuitions that researchers have acquired with respect to the usage of the softmax function, it is apparent that the understanding of its mathematical properties is still lacking. For instance, in the analysis of stateless multi-agent reinforcement learning schemes (\autoref{fig:Ctrl_RL_Sys}), when the action selection rule is taken as the softmax function, it is of interest which, if any, properties of softmax can allow us to conclude convergence of the learning algorithm towards a solution of the game (e.g., a Nash or logit equilibrium). Although the desired properties that can be used to conclude such convergence are fairly mundane, virtually no reference to these properties can be found within the existing body of literature. With regard to applications in the context of reinforcement and machine learning, the adjustment of the temperature constant of the softmax function is still performed on a rule-of-thumb basis. It has also been briefly speculated in \cite{Varying_Temp} that proper adjustment of the temperature constant can be used for game-theoretic reinforcement learning algorithms to achieve higher expected payoff. Therefore, an adaptive mechanism for scheduling the temperature constant would be desirable for many applications. Clearly, these questions can only be affirmatively answered by uncovering new properties of the softmax function.
	
	The goal of this paper is to expand on the known mathematical properties of the softmax function and demonstrate how they can be utilized to conclude the convergence of learning algorithm in a simple application of game-theoretic reinforcement learning. For additional examples and more involved applications, see our related paper \cite{Gao}. We perform our analysis and derive new properties by using tools from convex analysis \cite{Boyd, Wets} and monotone operator theory \cite{Bauschke, Facchinei}. It has been known that stateless multi-agent reinforcement learning that utilizes the softmax selection rule has close connections with the field of evolutionary game theory \cite{Bloembergen, Weiss, Coucheney_Learning, Merti_Learning, Sato, Kianercy, Tuyls}. Therefore, throughout this paper, we motivate some of the results through insights from the field of evolutionary game theory \cite{Price, Hofbauer_Sigmund, Weibull}.
	It is our hope that researchers across various disciplines can apply our results presented here to their domain-specific problems.

    The organization of this paper is as follows. Section II introduces notation convention for the rest of the paper. Section III introduces the definition of the softmax function, its different representations as well as a brief survey of several of its known properties from the existing literature. Section IV provides the background to convex optimization and monotone operator theory. In Section V, we derive additional properties of the softmax function. Section VI provides an analysis of a stateless continuous-time score-based reinforcement learning scheme within a single-player game setup to illustrate the application of these properties. Section VII provides the conclusion and some open problems for further investigation.

	\section{Notations}

	\noindent The notations used in this paper are as follows:
	
	\begin{itemize}[leftmargin = *]
		\item The $p$-norm of a vector is denoted as $\|\cdot\|_p$,  $1\leq p \leq \infty$.
		\item The $n-1$ dimensional unit simplex is denoted by $\Updelta^{n-1}$, where, $\Updelta^{n-1} \coloneqq \{x \in \mathbb{R}^n| \|x\|_1 = 1, x_i \geq 0\}.$
		\item The (relative) interior of $\Updelta^{n-1}$ is denoted by $\interior(\Updelta^{n-1})$, where, $\interior(\Updelta^{n-1}) \coloneqq \{x \in \mathbb{R}^n| \|x\|_1 = 1, x_i > 0\}.$
		\item $e_i \in \mathbb{R}^n$ denotes the $i^\text{th}$ canonical basis of $\mathbb{R}^n$, e.g., $e_i = \begin{bmatrix} 0, \ldots, 1, \ldots, 0 \end{bmatrix}^\top$, where $1$ occupies the $i^\text{th}$ position. 
		\item The vector of ones is denoted as $\mathbf{1} \coloneqq \begin{bmatrix} 1, \ldots, 1 \end{bmatrix}^\top$ and the vector of zeros is denoted as $\mathbf{0} \coloneqq \begin{bmatrix} 0, \ldots, 0 \end{bmatrix}^\top$. 
		\item Matrices are denoted using bold capital letters such as $\mathbf{A}$.
	\end{itemize}
	In general, a vector in the unconstrained space $\mathbb{R}^n$ will be denoted using $z$, while a vector in the $n-1$ dimensional unit simplex will be denoted using $x$. All logarithms are assumed to be base $e$.
	
	\section{Review of the Softmax Function and its Known Properties}
	While the \textit{softmax function} may take on different appearances depending on the application, its base model is that of a vector-valued function, whose individual component consists of an exponential evaluated at an element of a vector, which is normalized by the summation of the exponential of all the elements of that vector. In this section, we present several well-known and equivalent representations of the softmax function, and review some of its properties that are either immediate based on its definition or have been covered in the existing literature. 
	
	\subsection{Representations of the Softmax function} 
	
	The most well-known and widely-accepted version of the softmax function is as follows \cite{Goodfellow, Elfadel, Michalis, Naomi, Sparsemax, Tokic}.
	\begin{definition} The softmax function is given by $\sigma: \mathbb{R}^n \to \interior(\Updelta^{n-1})$,  	
	\begin{equation}
		\sigma(z) \coloneqq  \dfrac{1}{\textstyle\sum\limits_{j = 1}^n \exp(\lambda z_j)} \begin{bmatrix}\exp(\lambda z_1)  \\ \vdots \\ \exp(\lambda z_n) \end{bmatrix}, \lambda > 0,
		\label{def:softmax}
	\end{equation}
	where $\lambda$ is referred to as the \textit{inverse temperature constant}.
	\end{definition}
	\begin{remark} The softmax function is commonly presented in the literature as the individual components of \eqref{def:softmax}, 
	\begin{equation}
	\sigma_i(z) \coloneqq  \frac{\exp(\lambda z_i)}{\sum\limits_{j =1}^n \exp(\lambda z_j)}, 1 \leq i \leq n.
	\label{eqn:softmax_comp}
	\end{equation}
	\end{remark} 
	
	When $\lambda = 1$, we refer to \eqref{def:softmax} as the \textit{standard softmax function}. As $\lambda \to 0$, the output of $\sigma$ converges point-wise to the center of the simplex, i.e., a uniform probability distribution. On the other hand, as $\lambda \to \infty$, the output of $\sigma$ converges point-wise to $e_j \in \mathbb{R}^n$, where $j = \underset{1\leq i \leq n}{\text{argmax}}\thinspace e_i^\top z$, provided that the difference between two or more components of $z$ is not too small \cite{Merti_Learning, Elfadel}. We note that elsewhere in the literature, the reciprocal of $\lambda$ is also commonly used.

	\begin{remark} In $\mathbb{R}^2$, \eqref{eqn:softmax_comp} reduces to the \textit{logistic  function} in terms of $z_i - z_j$, \begin{equation} \sigma_i(z) =  \frac{\exp(\lambda z_i)}{\exp(\lambda z_i) + \exp(\lambda z_j)} =\frac{1}{1 + \exp(-\lambda (z_i-z_j))}, j \neq i \label{eqn:softmax_logistic}. \end{equation} Furthermore, we note that \eqref{eqn:softmax_comp} can be equivalently represented as, 
	\begin{equation} \sigma_i(z) = \exp(\lambda z_i - \log(\textstyle\sum_{j =1}^n \exp(\lambda z_j))).
	\label{eqn:softmax_rep_2}
	\end{equation} 
While \eqref{eqn:softmax_rep_2} is seldom used as a representation of the softmax function, in \cite{Harper}, the author noted that \eqref{eqn:softmax_rep_2} represents an \textit{exponential family}, which is the solution of the \textit{replicator dynamics} of evolutionary game theory \cite{Sandholm, Hofbauer_Sigmund, Weibull}. We will expand on the connections between the replicator dynamics and the softmax function in section V. 
	\end{remark}

	\begin{figure*}[htp!]
		\begin{center}
			\includegraphics[scale=0.80]{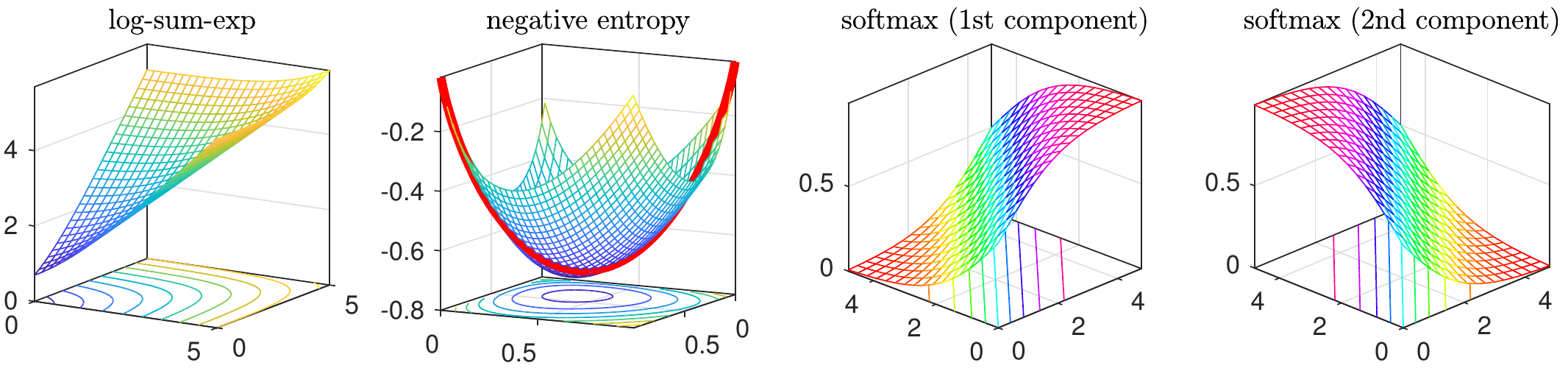}
			\caption{Plots of the log-sum-exp, negative entropy and both components of softmax function over $\mathbb{R}^2$ with $\lambda = 1$. The red curve on the negative entropy plot is the restriction of the negative entropy over the 1-dimensional simplex, $\Updelta^1$.}
			\label{fig:lse_nle_smx}
		\end{center}
	\end{figure*} 

	Another important representation of the softmax function can be derived by considering the ``argmax function" under entropy regularization.\footnote{As pointed out in \cite[p. 182]{Goodfellow}, the softmax function is a soft approximation of the argmax function, $z \mapsto \underset{x \in \Updelta^{n-1}}{\text{argmax}} \thinspace {x}^\top z$, not that of the ``max" function.} Let $z \in \mathbb{R}^n$, and consider the argmax of $x^\top z$ over the simplex,
	\begin{equation} 
	M(z) \coloneqq \underset{x \in \Updelta^{n-1}}{\text{argmax}} \thinspace {x}^\top z.
	\label{def:argmax_fcn}
	\end{equation}

	When there is a unique largest element in the vector $z$, it is clear that $M$ returns the basis vector corresponding to the entry of that element, that is, $M(z) = e_j$, where $j = \underset{1\leq i \leq n}{\text{argmax}}\thinspace e_i^\top z$. This solution corresponds to a vertex of the simplex. In general, however, \eqref{def:argmax_fcn} is set-valued; to see this, simply consider the case where two or more components of $z$ are equal.
	
	For many learning related applications, it is highly desirable for $M(z)$ to be singled-valued \cite{Coucheney_Learning, Merti_Learning, Naomi, Shalev, Hazan}. The most common approach to achieve this is by employing a so-called \textit{regularizer} function $\psi$ to \eqref{def:argmax_fcn}, which yields the \textit{regularized argmax function}:\footnote{Depending on the context, the regularizer is also referred to as an admissible deterministic perturbation \cite[p. 189]{Sandholm}, penalty function \cite{Coucheney_Learning, Merti_Learning}, smoothing function \cite{Leslie} or Bregman function \cite{Beck_Entropy}. For detailed construction of the regularizer, see \cite{Coucheney_Learning, Merti_Learning, Alvarez}.}
	\begin{equation} 
	\widetilde{M}(z) \coloneqq  \underset{x \in \Updelta^{n-1}}{\text{argmax}} \thinspace \left[{x}^\top z - \psi(x)\right].
	\label{eqn:Regularized_BR_map}
	\end{equation}
	
	A common choice of the regularizer is the \textit{negative entropy function restricted to the simplex}, which under the convention $0\log(0) = 0$, is given by $\psi:\mathbb{R}^n \to \mathbb{R}\cup\{+\infty\}$,
	\begin{equation} 
	\psi(x) \coloneqq \begin{cases} \lambda^{-1}\textstyle\sum\limits_{j =1}^n x_j \log(x_j), \lambda > 0 & x \in \Updelta^{n-1} \\
	   +\infty & x\notin \Updelta^{n-1}. \end{cases}	\label{def:neg_log_entropy}
	\end{equation} When $\lambda = 1$, we refer to \eqref{def:neg_log_entropy} as the \textit{standard negative entropy function.} 
	
	Since negative entropy is $\lambda^{-1}$-strongly convex\footnote{Recall that a function $f$ is $\mu$-strongly convex in $\|\cdot\|_{p}$ if there exists $\mu > 0$, s.t. $f(\theta z + (1-\theta)z^\prime) \leq \theta f(z) + (1-\theta)f(z^\prime) - \dfrac{\mu}{2}\theta(1-\theta)\|z - z^\prime\|_p^2$ for all $z,z^\prime \in \dom f$ and $\theta \in [0,1]$. $f$ is $\mu$-strongly concave if $-f$ is $\mu$-strongly convex.} in $\|\cdot\|_1$ over $\interior (\Updelta^{n-1})$ \cite{Beck_Entropy}, by strong concavity of the argument of \eqref{eqn:Regularized_BR_map}, it can be shown that by invoking the Karush-Kuhn-Tucker (KKT) conditions, the unique maximizer of \eqref{eqn:Regularized_BR_map} is the softmax function evaluated at $z \in \mathbb{R}^n$, i.e.,
	\begin{equation}
	 \underset{x \in \Updelta^{n-1}}{\text{argmax}} \thinspace [{x}^\top z - \lambda^{-1}\textstyle\sum\limits_{j =1}^n x_j \log(x_j)]  = \sigma(z).
	\label{eqn:softmax_rep_3}
	\end{equation}
	 It has been noted in \cite{Kianercy, Genewein, Rang, Shen_Dual} that the argument of the left-hand side of \eqref{eqn:softmax_rep_3},
	\begin{equation} {x}^\top z - \lambda^{-1}\textstyle\sum\limits_{j =1}^n x_j \log(x_j),
		\label{eqn:softmax_rep_5} \end{equation}
	represents the so called ``free energy" in statistical thermodynamics. In light of this connection, from a game-theoretic perspective, the softmax function can be thought of as providing the \textit{mixed strategy} with the maximum entropy which maximizes the payoff of a game \cite{Kianercy}. 
	
	It is also worth noting that the maximum of \eqref{eqn:softmax_rep_5} over the simplex is by definition the \textit{Legendre-Fenchel transform} of the negative entropy function \cite[p. 102]{Wets}, also commonly referred to as the \textit{log-sum-exp function}, which is given by $\lse: \mathbb{R}^n \to \mathbb{R} \cup \{+\infty\}$,
	\begin{equation}
		\label{def:log_sum_exp}
		\lse(z) \coloneqq  \lambda^{-1}\log(\textstyle\sum\limits_{j = 1}^n \exp(\lambda z_j)), \lambda > 0.
	\end{equation}
	When $\lambda = 1$, we refer to \eqref{def:log_sum_exp} as the \textit{standard log-sum-exp function}. 
	
	It is well-known that the log-sum-exp is an approximation to the \textit{vector-max} function \cite[p. 72]{Boyd}, \cite[p. 27]{Wets}, \[\vecmax(z) \coloneqq \max\{z_1, \ldots, z_n\}.\] That is, for any $z \in \mathbb{R}^n$, $\vecmax(z) \leq \lse(z) \leq \vecmax(z) + \lambda^{-1}\log(n)$, which can be shown by considering $\exp(\lambda \vecmax(z)) \leq \sum\limits_{j = 1}^n \exp(\lambda z_j) \leq n \exp(\lambda \vecmax(z)).$ Due to this reason, the log-sum-exp is sometimes referred to as the ``softmax function" in optimization-oriented literature. 
	
	We note that the dual or convex conjugate of the log-sum-exp function \eqref{def:log_sum_exp} is the negative entropy restricted to the simplex, given by \eqref{def:neg_log_entropy} \cite[p. 93]{Boyd}\cite[p. 482]{Wets}\cite{Shen_Dual}. 
	We illustrate the log-sum-exp function as well as the negative entropy and the softmax function in \autoref{fig:lse_nle_smx}. By \textit{Fenchel-Young} inequality, the log-sum-exp function is bounded below by a linear function,
	\begin{equation}
		\label{eqn:lse_ne_fenchel_young}
		\lse(z) \geq x^\top z - \psi(x), \forall x \in \Updelta^{n-1}, z \in \mathbb{R}^n.
	\end{equation}	
	Further consequences of the duality between the negative entropy and the log-sum-exp function as well as its role in game theory will not be explored at this time. Interested readers may refer to \cite{Elfadel2, Shen_Dual} or any standard textbooks on convex analysis, for example, \cite{Boyd, Wets, Hiriart-Urruty}.

	Finally, we provide a probabilistic characterization of the softmax function. Let $\epsilon_i, i \in \{1, \ldots, n\}$ be independent and identically distributed random variables with a \textit{Gumbel distribution} given by,
	\begin{equation}
		  \Pr[\epsilon_i \leq c] = \exp(-\exp(-\lambda c-\gamma)),
	\end{equation} 
	where $\gamma \approx 0.57721$ is the Euler-Mascheroni constant. It can be shown that for any vector $z \in \mathbb{R}^n$ \cite[p. 194]{Sandholm}\cite{Hof_Hop}, 
	\begin{equation}
	 \Pr\left[i = \underset{1 \leq j \leq n}{\text{argmax}} \thinspace z_j + \epsilon_j\right] =  \sigma_i(z).
	\label{eqn:softmax_rep_4}
	\end{equation}
	In game theory terms, \eqref{eqn:softmax_rep_4} represents the probability of choosing the \textit{pure strategy} that maximizes the payoff or score $z \in \mathbb{R}^n$, after the payoff or score has been perturbed by a stochastic perturbation.  
	\subsection{Properties of the Softmax - State of the Art}
	We  briefly comment on some properties of the softmax function that are either immediate or have been covered in the existing literature. First, $\sigma$ maps the origin of $\mathbb{R}^n$ to the barycenter of $\Updelta^{n-1}$, that is, $\sigma(\mathbf{0}) = n^{-1}\mathbf{1}$. The softmax function $\sigma$ is surjective but not injective, as it can easily be shown that for any $z, z+c\mathbf{1} \in \mathbb{R}^n$, $\forall c \in \mathbb{R}$, we have $\sigma(z+c\mathbf{1}) = \sigma(z)$. By definition, $\|\sigma(z)\|_1 = \sigma(z)^\top\mathbf{1} =  1, \forall z \in \mathbb{R}^n$. 
	
	In a recent paper, the authors of \cite{Sparsemax} noted that $\sigma(\mathbf{P}(z)) = \mathbf{P}\sigma(z)$, where $\mathbf{P}$ is any permutation matrix, and that the standard softmax function satisfies a type of ``coordinate non-expansiveness" property, whereby given a vector $z \in \mathbb{R}^n$, and suppose that $z_j \geq z_i$, then $0 \leq \sigma_j(z) - \sigma_i(z) \leq \dfrac{1}{2}(z_j - z_i)$. The last property can be derived by exploiting the properties of the hyperbolic tangent function. It was also noted that these properties of the softmax function bear similarities with the Euclidean projection onto $\Updelta^{n-1}$ \cite{Sparsemax}.

	In a direction that is tangential to the aim of this paper, the authors of \cite{Michalis} is interested in finding a bound on the softmax function. It can be shown that,
	\begin{equation}
		\sigma_i(z) = \frac{\exp(\lambda z_i)}{\textstyle\sum\limits_{j =1}^n \exp(\lambda z_j)} \geq \prod\limits_{\substack{j = 1\\ j\neq i}}^n \frac{1}{1+\exp(-\lambda(z_i - z_j))}, 
		\label{eqn:one_vs_each_bound}
	\end{equation}
	where \eqref{eqn:one_vs_each_bound} is referred as ``one-vs-each" bound, which can be generalized to bounds on arbitrary probabilities \cite{Michalis}. From \eqref{eqn:softmax_logistic}, we see that this inequality is tight for $n = 2$.
		
	\section{Review of Convex Optimization and Monotone Operator Theory}
    In this section we review some of the definitions and results from convex optimization and monotone operator theory that will be used in the derivation of new properties of the softmax function. Since the following definitions are standard, readers who are familiar with these subjects can skip this section without any loss of continuity. Most of the proofs of the propositions in this section can be found in references such as \cite{Boyd, Wets, Bauschke, Facchinei, Peypouquet, Hiriart-Urruty}. Throughout this section, we assume that $\mathbb{R}^n$ is equipped with the standard inner product $\langle z, z^\prime \rangle \coloneqq \sum\limits_{i =1}^n z_iz^\prime_i$ with the induced 2-norm $\|z\|_2\coloneqq \sqrt{\langle z, z \rangle}$. We assume the domain of $f$, $\dom f$, is convex. $C^1, C^2$ denote the class of continuously-differentiable and twice continuously-differentiable functions, respectively.
	
	\begin{definition} 
		A function $f: \dom f \subseteq \mathbb{R}^n \to \mathbb{R}$ is convex if,
			\begin{equation}
			f(\theta z + (1-\theta)z^\prime) \leq \theta f(z) + (1-\theta)f(z^\prime),
			\label{def:convex}
			\end{equation}
			for all $z,z^\prime \in \dom f$ and $\theta \in [0,1]$
		and strictly convex if \eqref{def:convex} holds strictly whenever $z \neq z^\prime$ and $\theta \in (0,1)$.
	\end{definition}
	
	The convexity of a $C^2$ function $f$ is easily determined through its Hessian $\nabla^2 f$. 
	\begin{lem} 
		\label{lem:convex_c2_characterization}
		Let $f$ be $C^2$. Then $f$ is convex if and only if $\dom f$ is convex and its Hessian is positive semidefinite, that is, for all $z \in \dom f, v \in \mathbb{R}^n$,
			\begin{equation}
			v^\top\nabla^2 f(z)v \geq 0,
			\end{equation}
		and strictly convex if $\nabla^2 f(z)$ is positive definite for all $z \in \dom f$. 
	\end{lem}
	
	Next, we introduce the concept of a monotone operator and its related properties. A monotone operator is usually taken as a set-valued relation, however, it is also natural for the definitions related to a monotone operator to be directly applied to single-valued maps \cite{Facchinei}.
	
	\begin{definition} (\cite[p. 154]{Facchinei})
		An operator (or mapping) $F: \mathcal{D} \subseteq \mathbb{R}^n \to \mathbb{R}^n$ is said to be:
		\begin{itemize}[leftmargin = *]
			\item pseudo monotone on $\mathcal{D}$ if, 
			\begin{equation}
				F(z^\prime)^\top(z-z^\prime) \geq 0 \implies F(z)^\top(z-z^\prime) \geq 0, \forall z,z^\prime \in \mathcal{D}.
			\end{equation}
			\item pseudo monotone plus on $\mathcal{D}$ if it is pseudo monotone on $\mathcal{D}$ and,
			\begin{equation}
			\begin{split}
			&F(z^\prime)^\top(z-z^\prime) \geq 0 \text{ and } F(z)^\top(z-z^\prime) = 0\\ \implies &F(z) = F(z^\prime), \forall z,z^\prime \in \mathcal{D}.
			\end{split}
			\end{equation}
			\item monotone on $\mathcal{D}$ if,
			\begin{equation}
			(F(z) - F(z^\prime))^\top(z-z^\prime) \geq 0, \forall z,z^\prime \in \mathcal{D}.
			\end{equation}
			\item monotone plus on  $\mathcal{D}$ if it is monotone on $\mathcal{D}$ and,
			\begin{equation}
			(F(z) - F(z^\prime))^\top(z-z^\prime) = 0 \implies F(z) = F(z^\prime), \forall z,z^\prime \in \mathcal{D}.
			\end{equation}
			\item strictly monotone on $\mathcal{D}$ if, 
			\begin{equation}
			(F(z) - F(z^\prime))^\top(z-z^\prime) > 0, \forall z,z^\prime \in \mathcal{D}, z \neq z^\prime.
			\end{equation}
		\end{itemize}
	\end{definition}
	Clearly, strictly monotone implies monotone plus, which in turn implies monotone, pseudo monotone plus and pseudo monotone. By definition, every strictly monotone operator is an injection.  We refer to an operator $F$ as being \textit{(strictly) anti-monotone} if $-F$ is (strictly) monotone. The following proposition provides a natural connection between $C^1$, convex functions and monotone gradient maps. 
	
	\begin{lem}
		\label{lem:convexity_and_monotonicity}
		A $C^1$ function $f$ is convex if and only if 
			\begin{equation}
			(\nabla f(z) - \nabla f(z^\prime))^\top(z - z^\prime) \geq 0, \forall z, z^\prime \in \dom f,
			\end{equation} and strictly convex if and only if, 
			\begin{equation}
			(\nabla f(z) - \nabla f(z^\prime))^\top(z - z^\prime) > 0, \forall z, z^\prime \in \dom f, z \neq z^\prime.
			\end{equation}
	\end{lem}
	Next, we introduce the notions of Lipschitz continuity and co-coercivity, and show that the two concepts are related through the gradient of a convex function. 
	
	\begin{definition} 
		An operator (or mapping) $F: \mathcal{D} \subseteq \mathbb{R}^n \to \mathbb{R}^n$ is said to be 
		\begin{itemize}[leftmargin=*]
			\item \textit{Lipschitz} (or \textit{$L$-Lipschitz}) if there exists a $L > 0$ such that,
			\begin{equation}
			\|F(z) - F(z^\prime)\|_2 \leq L \|z - z^\prime\|_2, \forall z, z^\prime \in \mathcal{D}.
			\label{def:L_Lipschitz}
			\end{equation}
			If $L = 1$ in \eqref{def:L_Lipschitz}, then $F$ is referred to as \textit{nonexpansive}. Otherwise, if $L \in (0,1)$, then $F$ is referred to as \textit{contractive}. 
			\item \textit{co-coercive} (or \textit{$\frac{1}{L}$-co-coercive}) if there exists a $L > 0$ such that, 
			\begin{equation}
			(F(z) - F(z^\prime)^\top(z-z^\prime ) \geq \frac{1}{L} \|F(z) - F(z)^\prime\|_2^2, \forall z, z^\prime \in \mathcal{D}.
			\label{def:co-coercive}
			\end{equation}
			If $L = 1$ in \eqref{def:co-coercive}, then $F$ is referred to as \textit{firmly nonexpansive}.
		\end{itemize}
	\end{definition}
	
	By Cauchy-Schwarz inequality, every $\frac{1}{L}$-co-coercive operator is $L$-Lipschitz, in particular, every firmly nonexpansive operator is nonexpansive. However, the reverse need not be true, for example $f(z) = -z$ is nonexpansive but not firmly nonexpansive. Fortunately, the Baillon-Haddad theorem (\cite[p. 40]{Peypouquet}, Theorem 3.13) provides the condition for when a $L$-Lipschitz operator is also $\frac{1}{L}$-co-coercive \cite{Baillon}. 
	
	\begin{thm} (Baillon-Haddad theorem)
		\label{thm:Baillon_Haddad}
		Let $f:\dom f \subseteq \mathbb{R}^n \to \mathbb{R}$ be a $C^1$, convex function on $\dom f$ and such that $\nabla f$ is L-Lipschitz continuous for some $L > 0$,  then $\nabla f$ is $\dfrac{1}{L}$-co-coercive.
	\end{thm}

	Finally, we will introduce the notion of \textit{maximal monotonicity}. Let $H: \mathbb{R}^n \to 2^{\mathbb{R}^n}$ be the set-valued map, where $2^{\mathbb{R}^n}$ denotes the power set of $\mathbb{R}^n$. Let the graph of $H$ be given by $\text{gra}H \coloneqq \{(u,v) \in \mathbb{R}^n \times \mathbb{R}^n| v = Hu\}$. The set-valued map $H$ is said to be \textit{monotone} if $
	(u - u^\prime)^\top(v-v^\prime) \geq 0, v \in H(u), v^\prime \in H(u^\prime)$. 
	\begin{definition} (\cite[p. 297]{Bauschke}) Let $H: \mathbb{R}^n \to 2^{\mathbb{R}^n}$ be monotone. Then $H$ is maximal monotone if there exists no monotone operator $G:\mathbb{R}^n \to 2^{\mathbb{R}^n}$ such that $\gra{G}$ properly contains $\gra{H}$, i.e., for every $(u,v) \in \mathbb{R}^n \times \mathbb{R}^n$, 
	\begin{equation}
		(u,v)\in \gra H \Leftrightarrow (\forall (u^\prime, v^\prime) \in \gra H) \thinspace (u - u^\prime)^\top(v - v^\prime) \geq 0.
	\end{equation}
	\end{definition}
	By Zorn's Lemma, every monotone operator can be extended to a maximal monotone operator \cite[p. 535]{Wets}, \cite[p. 297]{Bauschke}. For the scope of this paper, we are interested when a single-valued map is maximal monotone. The following proposition provides a simple characterization of this result \cite[p. 535]{Wets}. 
	\begin{lem}
		If a continuous mapping $F: \mathbb{R}^n \to \mathbb{R}^n$ is monotone, it is maximal monotone. In particular, every differentiable monotone mapping is maximal monotone.
		\label{lem:maximal_monotone_characterization}
	\end{lem}

	\section{Derivation of Properties of Softmax function} 
	
	In this section we derive several properties of the softmax function using tools from convex analysis and monotone operator theory introduced in the previous section. We begin by establishing the connection between the log-sum-exp function and the softmax function. 
	
	It has long been known that the softmax function is the gradient map of a convex potential function \cite{Elfadel}, however, the fact that its potential function is the log-sum-exp function (i.e., \eqref{def:log_sum_exp}) is rarely discussed.\footnote{Although not explicitly stated, this relationship could also be found in \cite[p. 93]{Boyd} and various other sources.} We make this connection clear with the following proposition.
	
	\begin{prop}
		\label{prop:gradient_of_lse_is_softmax}
		The softmax function is the gradient of the log-sum-exp function, that is, $\sigma(z) = \nabla \lse(z)$. 
	\end{prop}
	
	\begin{proof}
		
		Evaluating the partial derivative of $\lse$ at each component yields $\dfrac{\partial \lse(z)}{\partial z_i} = \dfrac{\exp(\lambda z_i)}{\textstyle{\sum}_{j = 1}^n \exp(\lambda z_j)}$.
			By definition of the gradient, we have,
	\[
\nabla \lse(z) = \begin{bmatrix} \dfrac{\partial \lse(z)}{\partial z_1}  \\ \vdots \\ \dfrac{\partial \lse(z)}{\partial z_n} \end{bmatrix} = \dfrac{1}{\sum\limits_{j = 1}^n \exp(\lambda z_j)} \begin{bmatrix}  \vphantom{\dfrac{\partial \lse(z)}{\partial z_1}}\exp(\lambda z_1)  \\ \vdots \\ \exp(\lambda z_n) \vphantom{\dfrac{\partial \lse(z)}{\partial z_n}}\end{bmatrix} = \sigma(z).
\]
	\end{proof}

	Next, we calculate the Hessian of the log-sum-exp function (and hence the Jacobian of the softmax function). 
	\begin{prop}
		\label{prop:softmax_jacobian}
		The Jacobian of the softmax function and Hessian of the log-sum-exp function is given by: 
		\begin{equation}
		\Jacob[\sigma(z)] = \nabla^2 \lse(z) = \lambda(\diag(\sigma(z)) - \sigma(z)\sigma(z)^\top),
		\label{eqn:Hessian_of_lse}
		\end{equation}
		where \eqref{eqn:Hessian_of_lse} is a symmetric positive semidefinite matrix and satisfies $\Jacob[\sigma(z)]\mathbf{1} = \mathbf{0}$, that is, $\mathbf{1}$ is the eigenvector associated with the zero eigenvalue of $\Jacob[\sigma(z)]$.
	\end{prop}
	
	\begin{proof}
		
		The diagonal entries of $\nabla^2 \lse$ are given by,
		\[
		\dfrac{\partial^2 \lse(z)}{\partial z_i^2} = \dfrac{\lambda\left[\exp(\lambda z_i)\sum\limits_{j = 1}^n \exp(\lambda z_j) - \exp(\lambda z_i)^2\right]}{(\sum\limits_{j = 1}^n \exp(\lambda z_j))^2},
		\]
		and the off-diagonal entries of $\nabla^2 \lse$ are given by the mixed partials,
		\[
		\dfrac{\partial^2 \lse(z)}{\partial z_k\partial z_i} = \dfrac{-\lambda \exp(\lambda z_k)\exp(\lambda z_i)}{(\sum\limits_{j = 1}^n \exp(\lambda z_j))^2}.
		\]
		Assembling the partial derivatives, we obtain the Hessian of $\lse$ and the Jacobian of $\sigma$:
		\begin{equation}
		\Jacob[\sigma(z)] = \nabla^2 \lse(z) = \lambda(\diag(\sigma(z)) - \sigma(z)\sigma(z)^\top).
		\end{equation}
		
		The symmetry of $\Jacob[\sigma(z)]$ comes from the symmetric structure of the diagonal and outer product terms. The positive semi-definiteness of $\Jacob[\sigma(z)]$ follows from an application of the Cauchy-Schwarz inequality \cite[p. 74]{Boyd}. It can be shown through direct computation that $\Jacob[\sigma(z)]\mathbf{1} = \mathbf{0}$ or alternatively refer to \cite[p. 213]{Sandholm}.
	\end{proof}
	\begin{remark}
	This result was previous noted in references such as \cite{Elfadel, Elfadel2} and can be found in \cite[p. 195]{Sandholm}\cite[p. 74]{Boyd}. As a trivial consequence of \autoref{prop:softmax_jacobian}, we can write the individual components of $\Jacob[\sigma(z)]$ as, \begin{equation}\Jacob_{ij}[\sigma(z)] = \lambda\sigma_i(z)(\delta_{ij} - \sigma_j(z)),\end{equation} where $\delta_{ij}$ is the Kronecker delta function. This representation is preferred for machine learning related applications and is loosely referred to as the ``derivative of the softmax" \cite{Alpaydin}.  
	\end{remark} 
	\begin{remark}
	Using the Jacobian of the softmax function given in \eqref{eqn:Hessian_of_lse}, we provide the following important observation that connects the field of evolutionary game theory with convex analysis and monotone operator theory. Let $x = \sigma(z)$, then we have,
	\begin{equation}
	\left. \nabla^2 \lse(z) \right|_{x = \sigma(z)} = \lambda(\diag(x) - xx^\top).
	\end{equation}
	We note that this is precisely the matrix term appearing in the replicator dynamics \cite[p. 229]{Sandholm}, \cite{Projection}, that is, 
	\begin{equation}
		\dot x = 	\left. \nabla^2 \lse(z) \right|_{x = \sigma(z)}u = \lambda(\diag(x) - xx^\top)u,
	\end{equation}
	where $x \in \Updelta^{n-1}$ is a mixed strategy and $u \in \mathbb{R}^n$ is a \textit{payoff vector}. We note that the matrix term was referred to as the \textit{replicator operator} in \cite{Hopkins_Learning}. To the best of our knowledge, the implications of this connection has not been discussed in the evolutionary game theory community. 
	\end{remark}
	
	\begin{restatable}{lem}{Logsumexpisconvex}
		\label{lem:logsumexpisconvex}	
		The log-sum-exp function is $C^2$, convex and not strictly convex on $\mathbb{R}^n$.
	\end{restatable}
	
	The convexity of the log-sum-exp function is well-known \cite{Boyd} and follows from \autoref{prop:softmax_jacobian}. 
	To show that log-sum-exp is not strictly convex, take $z$ and $z + c\mathbf{1}$, where $z\in \mathbb{R}^n, c \in \mathbb{R}$, then, 
	\begin{equation}
	\label{eqn:lse_is_affine}
	\lse(z + c\mathbf{1}) = \lse(z) + c.
	\end{equation}
	Thus, $\lse$ is affine along the line given by $z + c\mathbf{1}$, which implies that the log-sum-exp function is not strictly convex.  This result is also noted in \cite[p. 48]{Wets}. 

	\begin{prop}
		\label{prop:softmax_is_monotone}
		
		The softmax function is monotone, that is,	
		\begin{equation}
			(\sigma(z) - \sigma(z^\prime))^\top(z - z^\prime) \geq 0, \forall z, z^\prime \in \mathbb{R}^n,
			\label{eqn:softmax_is_monotone}
		\end{equation}	
		and not strictly monotone on $\mathbb{R}^n$.
	\end{prop}
	
	\begin{proof}
		Monotonicity of $\sigma$ follows directly from the convexity of the log-sum-exp function. Since the log-sum-exp function is not strictly convex on $\mathbb{R}^n$, therefore by \autoref{lem:convexity_and_monotonicity}, $\sigma$ fails to be strictly monotone. Alternatively, since every strictly monotone operator is injective, therefore $\sigma$ is not strictly monotone on $\mathbb{R}^n$.   
	\end{proof}
	
	The monotonicity of $\sigma$ allows us to state a stronger result.
	
	\begin{cor}
		\label{cor:softmax_is_maximal_monotone}
		The softmax function is a maximal monotone operator, that is, there exists no monotone operator such that its graph properly contains the graph of the softmax function.
	\end{cor}
	
	\begin{proof}
		This directly follows from $\sigma$ being a continuous, monotone map, see \autoref{lem:maximal_monotone_characterization}. 
	\end{proof}

	Next, we show that under appropriate conditions, the softmax function is a contraction in $\|\cdot\|_2$.

	\begin{lem}(\cite[p. 58]{Nesterov}, Theorem 2.1.6)	
		\label{lem:lipschitzcontinuousgradient2} A $C^2$, convex function $f:\mathbb{R}^n \to \mathbb{R}$ has a Lipschitz continuous gradient with Lipschitz constant $L > 0$ if for all $z, v \in \mathbb{R}^n$,
		\begin{equation}
			0 \leq v^\top \nabla^2 f(z)v \leq L\|v\|^2_2.
			\label{eqn:hessian_lip_cont_grad_condition}
		\end{equation}
	\end{lem} 
	
	\begin{prop}
		\label{prop:softmax_is_lipschitz}
		The softmax function is $L$-Lipschitz with respect to $\|\cdot\|_2$ with $L = \lambda$, that is, for all $z, z^\prime \in \mathbb{R}^n$,
		\begin{equation}
			\|\sigma(z) - \sigma(z^\prime)\|_2 \leq \lambda\|z - z^\prime\|_2,
			\label{eqn:softmax_is_lipschitz}
		\end{equation}
		where $\lambda$ is the inverse temperature constant.
	\end{prop}
	\begin{proof}
		Given the Hessian of $\lse$ in \autoref{prop:softmax_jacobian}, we have for all $z, v \in \mathbb{R}^n,$
		\begin{align*}
			v^\top \nabla^2\lse(z)v = & \lambda(\textstyle\sum\limits_{i = 1}^n v_i^2\sigma_i(z) - (\sum\limits_{i= 1}^n v_i\sigma_i(z))^2).
			\numberthis \label{eqn:hess_of_lse_expand}
			\shortintertext{Since the second term on the right hand side of \eqref{eqn:hess_of_lse_expand} is nonnegative, therefore,}
			v^\top \nabla^2\lse(z)v \leq & \lambda\textstyle\sum\limits_{i = 1}^n v_i^2\sigma_i(z) \leq \lambda\sup\{\sigma_i(z)\}\sum\limits_{i = 1}^n v_i^2\\
			\implies & v^\top \nabla^2\lse(z)v \leq \lambda \|v\|^2_2. \numberthis
			\label{eqn:hess_is_pos_def_1}
		\end{align*}
		where $\sup\{\sigma_i(z)\} = 1, \forall i \in \{1, \ldots, n\}, \forall z \in \mathbb{R}^n$. 
		By \autoref{lem:logsumexpisconvex}, $\nabla^2 \lse(z)$ is positive semidefinite. Hence using \autoref{lem:convex_c2_characterization} and \eqref{eqn:hess_is_pos_def_1}, we have,
		\begin{equation}
			0 \leq v^\top\nabla^2 \lse(z)v \leq \lambda\|v\|^2_2.
		\end{equation}
		By \autoref{lem:lipschitzcontinuousgradient2}, $\sigma$ is Lipschitz with $L = \lambda$.
	\end{proof}

	We note that \autoref{prop:softmax_is_lipschitz} can also be established by using Theorem 4.2.1. in \cite[p. 240]{Hiriart-Urruty}, which resorts to using duality between the negative entropy and the log-sum-exp function. 
	
	As a minor consequence of \autoref{prop:softmax_is_lipschitz}, by the Cauchy-Schwarz inequality, we have,
	\begin{equation}(\sigma(z) - \sigma(z^\prime))^\top(z-z^\prime) \leq \lambda\|z - z^\prime\|_2^2. \end{equation}

	\begin{cor}
		
		\label{cor:softmax_is_co_coercive}
		The softmax function is $\frac{1}{L}$-co-coercive with respect to $\|\cdot\|_2$ with $L = \lambda$, that is, for all $z, z^\prime \in \mathbb{R}^n$,
		\begin{equation}
			(\sigma(z) - \sigma(z^\prime))^\top(z-z^\prime) \geq \dfrac{1}{\lambda}\|\sigma(z) - \sigma(z^\prime)\|_2^2,
			\label{eqn:softmax_is_co_coercive}
		\end{equation}
		where $\lambda$ is the inverse temperature constant.
	\end{cor}
	\begin{proof}
		Follows directly from Baillon - Haddad Theorem, see \autoref{thm:Baillon_Haddad}. 
	\end{proof}

	\autoref{prop:softmax_is_lipschitz} and \autoref{cor:softmax_is_co_coercive} show that the inverse temperature constant $\lambda$ is crucial in determining the Lipschitz and co-coercive properties of the softmax function. We summarize these properties with respect to $\|\cdot\|_2$ in the following corollary. 
	
	\begin{cor}
		The softmax function is $\lambda$-Lipschitz and $\frac{1}{\lambda}$-co-coercive for any $\lambda > 0$, in particular,
		\begin{itemize}[leftmargin = *]
			\item Nonexpansive and firmly nonexpansive for $\lambda = 1$,
			\item Contractive for $\lambda \in (0,1)$,
		\end{itemize}
		where $\lambda$ is the inverse temperature constant.
		\label{eqn:softmax_lipschitz_property}
	\end{cor}

	Finally, we state an additional consequence of $\sigma$ being a Lipschitz, monotone operator with a symmetric Jacobian matrix over all of $\mathbb{R}^n$. 
	\begin{cor}
		The softmax function is monotone plus, pseudo monotone plus and pseudo monotone on $\mathbb{R}^n$. 
	\end{cor}

	\begin{proof}
		This follows from the chain of implications in \cite[p. 164]{Facchinei}.
	\end{proof}

	\section{Application in Game-Theoretic Reinforcement Learning} 
		\begin{figure}[h]
		\begin{center}
			\includegraphics[scale=0.8]{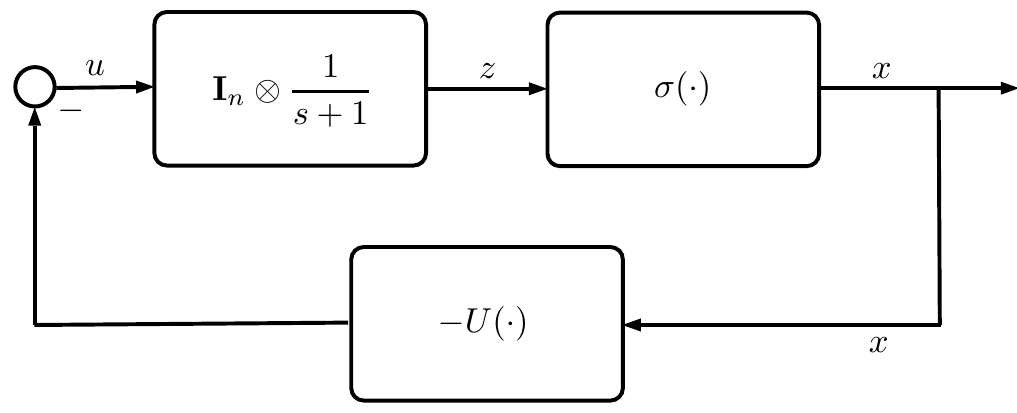}
		\end{center}
		\caption{Feedback representation of the exponentially-discounted reinforcement learning scheme (EXP-D-RL).} 
		\label{fig:Z_Exponential_Learning_2}
	\end{figure}
	
	In this section we demonstrate an application of these new properties of the softmax function in the context of stateless continuous-time reinforcement learning in finite games. For clarity and ease of notation, we perform our analysis in a single-player setup. For extension to $N$-player games, higher-order extension, and addition simulations, refer to our related paper \cite{Gao}. For other related work in this direction, see \cite{Coucheney_Learning, Merti_Learning, Laraki_Higher_Game}. 
	
	Consider a game $\mathcal{G}$ with a single player. We note that type of game is also known as ``play against nature" and is identifiable with single-population matching in population games \cite{Sandholm}. The player is equipped with an action set $\mathcal{A} = \{1, \ldots, n\}$ and continuous payoff function $\mathcal{U}: \mathcal{A} \to \mathbb{R}$. A mixed strategy profile is given by $x = \begin{bmatrix} x_1, \ldots, x_n \end{bmatrix}^\top \in \Updelta^{n-1}$. The player's expected payoff of using $x$ is given by,
	\[
		\mathcal{U}(x) = \sum\limits_{i \in \mathcal{A}} x_i U_i(x) = x^\top U(x),
	\]
	where $u = U(x) = \begin{bmatrix} U_1 & \ldots U_n \end{bmatrix}^\top \in \mathbb{R}^n$ is referred to the \textit{payoff vector} at $x$. 

	Starting at $t = 0$, we assume that the player repeatedly interacts with the game and aggregates his raw payoff $u = U(x) \in \mathbb{R}^n$ into score variables (or \textit{Q-values}) $z \in \mathbb{R}^n$ via the learning rule,
	\begin{equation}
	z_i(t) = e^{-t} z_i(0) + \small\int\limits_0^t e^{-(t-\tau)} u_i(\tau) \mathrm{d}\tau, \forall i \in \mathcal{A},
	\label{def:exponential_discount_learning}
	\end{equation}
	where $u_i = U_i(x) \in \mathbb{R}$ is the payoff to the $i^{\text{th}}$ strategy and $z_i \in \mathbb{R}$ is the score variable associated with the $i^{\text{th}}$ strategy. This form of aggregation as given by \eqref{def:exponential_discount_learning} is known as \textit{exponentially-discounted learning rule}, under which the player allocates exponentially more weight to recent observations of the payoff \cite{Coucheney_Learning, Merti_Learning}. 
	
	Taking the time derivative of \eqref{def:exponential_discount_learning} yields the \textit{score dynamics},
	\begin{equation}
	\dot z_i = u_i - z_i,  \forall i \in \mathcal{A},
	\label{eqn:first_order_score_dyn_comp}
	\end{equation}
	We refer to \eqref{eqn:first_order_score_dyn_comp} as the \textit{exponentially-discounted score dynamics}, a set of differential equations whose solutions capture the evolution of the player's scores over time. This form of score dynamics was investigated in \cite{Kianercy, Coucheney_Learning, Sato, Cominetti}. Since $U_i(x)$ is continuous over a compact domain, therefore there exists some constant $M > 0$ such that $|U_i(x)| \leq M, \forall x \in \Updelta^{n-1}$. Then it can be shown using standard arguments that $|z_i(t)| \leq \max\{|z_i^p(t)|, M\}, \forall t \geq 0$ and $\Omega = \{z\in \mathbb{R}^n| \|z\| \leq \sqrt{n}M\}$ is a compact, positively invariant set (solution remains in $\Omega$ for all time).
	
	We can express \eqref{eqn:first_order_score_dyn_comp} using stacked-vector notation as,
	\begin{equation}
	\dot z = u - z,
	\label{eqn:first_order_score_dyn}
	\end{equation} where $z = \begin{bmatrix} z_1, \ldots, z_n\end{bmatrix}^\top$.  Suppose that the score variable $z$ is mapped to the strategy $x$ from the softmax selection rule, i.e., $x = \sigma(z)$, then the payoff vector can be written as $u = U(x) =  U(\sigma(z))$. Expressing the composition between the softmax selection rule with the payoff vector as $U \circ \sigma(z) \coloneqq U(\sigma(z))$, then we can also write \eqref{eqn:first_order_score_dyn} as,
	\begin{equation}
	\dot z = (U\circ\sigma)(z)  - z.
	\label{Exponential_Discount_Dynamic_2_Overall}
	\end{equation}
	
	The overall \textit{exponentially-discounted reinforcement learning scheme} (EXP-D-RL) can be represented as a closed-loop feedback system in \autoref{fig:Z_Exponential_Learning_2}, where $\mathbf{I}_{n}$ is the $n \times n$ identity matrix, $\otimes$ is the Kronecker product, and $\dfrac{1}{s+1}$ is the \textit{transfer function} of \eqref{eqn:first_order_score_dyn_comp} from $u_i$ to $z_i, s \in \mathbb{C}$. The closed-loop system is equivalently represented by,
	\begin{equation} 
	\begin{cases}
	\dot z = u - z,\\
	{u} = U(x),\\
	{x} = \sigma(z).
	\end{cases}
	\label{eqn:first_order_exponentially_discounted_dynamics_system}
	\end{equation}

	From \eqref{Exponential_Discount_Dynamic_2_Overall}, we see that the equilibria of the overall closed-loop system  \eqref{eqn:first_order_exponentially_discounted_dynamics_system} are the fixed points of the map $z \mapsto (U\circ\sigma)(z)$. This fixed-point condition can be restated as,
	\begin{equation} 
	\begin{cases}
	\dot z = 0 \implies \overline{u}^\star = \overline{z}^\star,\\
	{\overline{u}^\star} = U(\overline{x}^\star),\\
	{\overline{x}^\star} = \sigma(\overline{z}^\star).
	\end{cases}
	\label{eqn:fixed_point_condition_score_dyn}
	\end{equation}
%
%

	The existence of the fixed point is guaranteed by the Brouwer's Fixed Point Theorem provided that $U \circ \sigma$ is a continuous function with bounded range \cite{Cominetti}. Since $\overline{z}^\star = \overline{u}^\star$, therefore the fixed point $\overline{z}^\star$ is mapped through $\sigma$ to a logit equilibrium \cite{qre, Coucheney_Learning}. 
	
	\begin{prop}
		\label{prop:qre} $\overline{x}^\star = \sigma(\overline{z}^\star) = \sigma(\overline{u}^\star)$ is the logit equilibrium of the game $\mathcal{G}$. 
	\end{prop}
	
	Hence, the convergence of the solution of the score dynamics $z(t)$ towards the fixed point of $U \circ \sigma$ implies convergence of the induced strategy $x(t) = \sigma(t)$ towards a logit equilibrium point $\overline{x}^*$ of the game. In the following, we provide different assumptions on the payoff function $U$ or the composition between the payoff function and the softmax operator $U \circ \sigma$ under which the induced strategy converges.   For background on dynamical systems and Lyapunov theory, see \cite{Khalil}.  This analysis was inspired by \cite{Cominetti}.

	First, we exploit the co-coercive property of the softmax function to provide the convergence conditions of the exponentially-discounted score dynamics \eqref{eqn:first_order_score_dyn} in a general class of games. Consider the exponentially-discounted reinforcement learning scheme as depicted in \autoref{fig:Z_Exponential_Learning_2}. We proceed by imposing the following assumption on the payoff of the game.

	\begin{assumption}
		\label{assump:payoff_monotonicity}
		The payoff $U$ is anti-monotone, that is, for all $x, x^\prime \in \Updelta^{n-1}$,
		\begin{equation}
		\label{eqn:U_Anti_Monotone_Component}
		(x-{x}^\prime)^\top(U(x) - U({x}^\prime)) \leq 0.
		\end{equation}
	\end{assumption} 
	
	\begin{thm}
		\label{thm:stable_game_convergence}
		 Let $\mathcal{G}$ be a game with player's learning scheme as given by EXP-D-RL, \eqref{eqn:first_order_exponentially_discounted_dynamics_system} (\autoref{fig:Z_Exponential_Learning_2}). Assume there are a finite number of isolated fixed-points $\overline{z}^\star$ of $U \circ \sigma$, then under \autoref{assump:payoff_monotonicity}, the player's score $z(t)$ converges to a rest point $\overline{z}^\star$. Moreover, $x(t) = \sigma(z(t))$ converges to a logit equilibrium $\overline{x}^\star = \sigma(\overline{z}^\star)$ of $\mathcal{G}$.
	\end{thm}
	
	\begin{proof}
		First, recall that solutions $z(t)$ of remain bounded and $\Omega = \{z \in \mathbb{R}^n|\|z\|_2 \leq \sqrt{n}M\}$ is a compact, positively invariant set. Let $z^\star$ be a rest point, $\overline{z}^\star = \overline{u}^\star = U(\sigma(\overline{z}^\star))$, $\overline{x}^\star = \sigma(\overline{z}^\star)$. 		
		
		Next, consider the Lyapunov function given by the Bregman divergence generated by the log-sum-exp function \eqref{def:log_sum_exp},
		\begin{equation}
		\label{eqn:lse_lyapunov_fcn}
		V_{\overline{z}^\star}(z) = \lse(z) - \lse(\overline{z}^\star) - \nabla\lse(\overline{z}^\star)^\top(z-\overline{z}^\star),
		\end{equation}
		 Recall that by \autoref{lem:logsumexpisconvex}, $\lse$ is convex and by \autoref{prop:gradient_of_lse_is_softmax}, $\nabla\lse(z) = \sigma(z)$. By convexity of $\lse$, $V_{\overline{z}^\star}(z) \geq 0, \forall z \in \mathbb{R}^n$. Using $\|\sigma(z)\|_1 = \sigma(z)^\top\mathbf{1} = 1$ and $\lse(z + \mathbf{1}c) = \lse(z) + c$, it can be shown that $V_{\overline{z}^\star}(\overline{z}^\star + \mathbf{1}c) = 0, \forall c \in \mathbb{R}$, so $V_{\overline{z}^\star}(\cdot)$ is positive semidefinite, but not positive definite. 
		
		Taking the time derivative of $V_{z^\star}(z)$ along the solution of \eqref{Exponential_Discount_Dynamic_2_Overall} yields,
		\begin{align*}
		\dot V_{\overline{z}^\star}(z) = & \nabla V_{\overline{z}^\star}(z)^\top\dot z \\ =& (\sigma(z) - \sigma(\overline{z}^\star))^\top(-z+u) \\
		= &(\sigma(z) - \sigma(\overline{z}^\star))^\top(-z + \overline{z}^\star - \overline{z}^\star + u) \\
		= & -(\sigma(z) - \sigma(\overline{z}^\star))^\top(z - \overline{z}^\star) +  (\sigma(z) - \sigma(\overline{z}^\star))^\top(u - \overline{u}^\star).
		\shortintertext{By \autoref{cor:softmax_is_co_coercive}, $\sigma$ is co-coercive, therefore,}
		\dot V_{\overline{z}^\star}(z)  \leq &   -\frac{1}{\lambda}\|\sigma(z) - \sigma(\overline{z}^\star)\|_2^2 + (\sigma(z) - \sigma(\overline{z}^\star))^\top(u - \overline{u}^\star).
		\end{align*}
		
		Since $u = U(\sigma(z)), \overline{u}^\star = U(\sigma(\overline{z}^\star))$, $x = \sigma(z)$, and $x^\star = \sigma(\overline{z}^\star)$, therefore \eqref{eqn:U_Anti_Monotone_Component} implies that $\dot V_{\overline{z}^\star}(z)  \leq   -\frac{1}{\lambda}\|\sigma(z) - \sigma(\overline{z}^\star)\|_2^2$, thus $\dot{V}_{\overline{z}^\star}(z) \leq 0, \forall z \in \mathbb{R}^n$, and $\dot V_{\overline{z}^\star}(z) = 0$, for all $z \in \mathcal{E} = \{z \in \Omega| \sigma(z) = \sigma(\overline{z}^\star)\}$. On $\mathcal{E}$ the dynamics of \eqref{Exponential_Discount_Dynamic_2_Overall} reduces to, 
		\[
		\dot z = U(\sigma(\overline{z}^\star)) - z = \overline{z}^\star - z.
		\]
		Therefore $z(t) \to \overline{z}^\star$ as $t \to \infty$, for any $z(0) \in \mathcal{E}.$ Thus, no other solution except $\overline{z}^\star$ can stay forever in $\mathcal{E}$, and the largest invariant subset $\mathcal{M} \subseteq \mathcal{E}$ consists only of equilibria. Since \eqref{Exponential_Discount_Dynamic_2_Overall} has a finite number of isolated equilibria $\overline{z}^\star$, by LaSalle's invariance principle \cite{Khalil}, it follows that for any $z(0) \in \Omega$, $z(t)$ converges to one of them. By continuity of $\sigma$, $x(t)$ converges to $\overline{x}^\star = \sigma(\overline{z}^\star)$ as $t \to \infty$. For an alternative proof using Barbalat's lemma, see \cite{Gao}. 		
	\end{proof} 
	\begin{figure}[h]
		\begin{center}
			\includegraphics[scale=0.6]{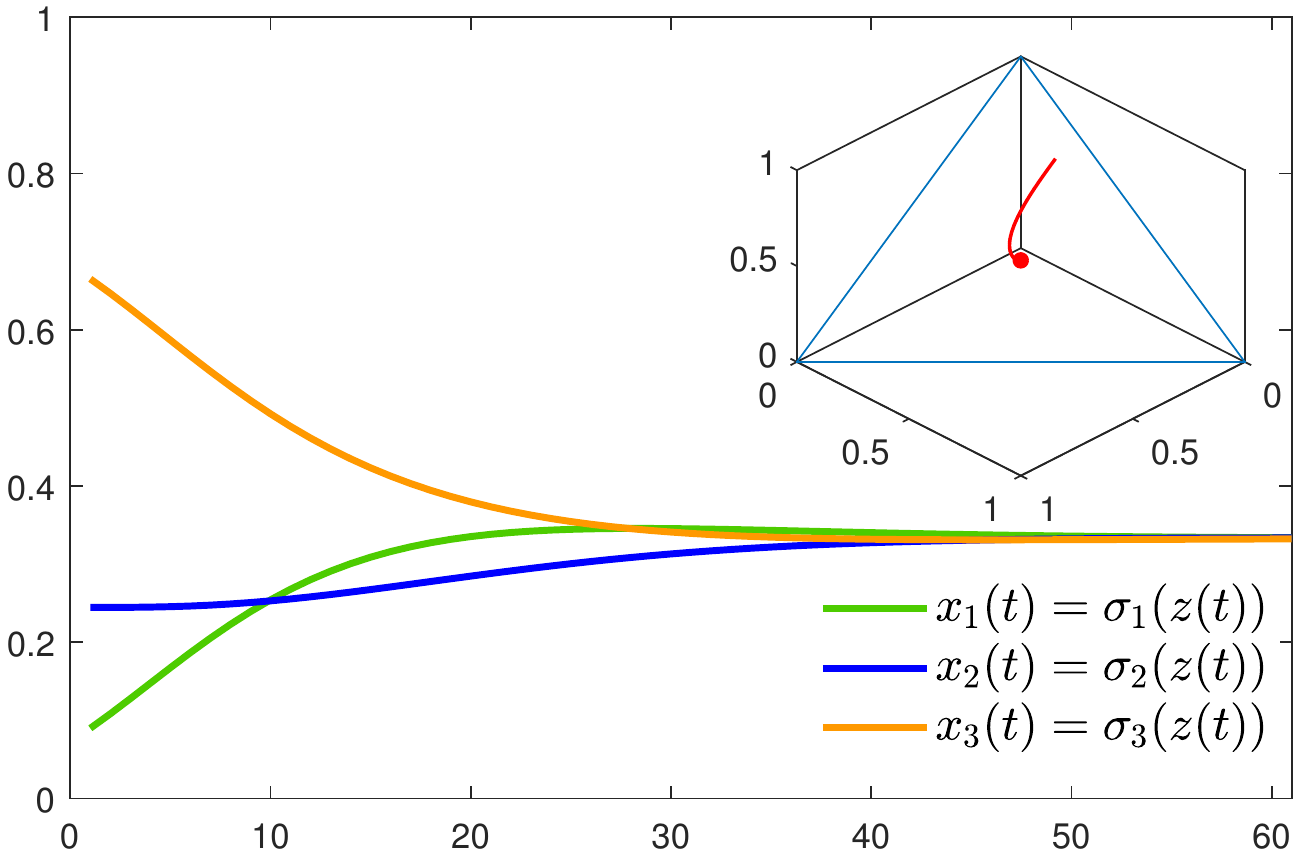}
		\end{center}
		\caption{Convergence of the induced strategy $x(t)$ towards the logit equilibrium of the standard RPS game. The red curve shows the evolution of the strategy in the interior of the simplex.}
		\label{fig:RPS_Game}
	\end{figure}
	\begin{example} 
	We note that \autoref{assump:payoff_monotonicity} is equivalent to game $\mathcal{G}$ being a \textit{stable game} \cite[p. 79]{Sandholm}, \cite{Stable}. The representative game from the class of stable games is the standard Rock-Paper-Scissors (RPS) game given by the payoff matrix,
	\begin{equation}
		\label{eqn:RPS_game}
		\mathbf{A} = \begin{bmatrix*}[r] 0 & -1 & 1 \\ 1 & 0 & -1 \\ -1 & 1 & 0 \end{bmatrix*},
	\end{equation}
	which generates the payoff vector $U(x) = \mathbf{A}x$. We present a simulation of the standard RPS game under the exponentially-discounted score dynamics \eqref{eqn:first_order_score_dyn} with $\lambda = 1$. The resulting induced strategy $x(t)$ is shown in \autoref{fig:RPS_Game}, which by \autoref{thm:stable_game_convergence} (which uses the co-coercivity property of the softmax function) is guaranteed to converge to the logit equilibrium of the RPS game, which is given by $\overline{x}^\star = \begin{bmatrix} 1/3 & 1/3 & 1/3 \end{bmatrix}^\top.$ In this game, the logit equilibrium coincides with the Nash equilibrium.  
	\end{example} 
	Next, we rely on a slightly modified result in \cite{Cominetti} to show that the Lipschitzness of the softmax function can be directly used to conclude the convergence of the score dynamics \eqref{eqn:first_order_score_dyn} for certain classes of games. 
	
		\begin{assumption}
		\label{assump:composite_map_contraction}
		$U \circ \sigma$ is $\|\cdot\|_\infty$-contractive, that is, there exists a constant $L \in (0,1)$ such that for all score variables $z, z^\prime \in \mathbb{R}^n$,
		\begin{equation}
		\|(U\circ\sigma)(z) - (U\circ\sigma)(z^\prime)\|_\infty \leq L \|z-z^\prime\|_\infty.
		\label{eqn:inf_contractive_assumption}
		\end{equation}
	\end{assumption}

	\begin{prop}{(Theorem 4, \cite{Cominetti})}
		\label{prop:cominetti_norm_contraction} Under \autoref{assump:composite_map_contraction}, the unique fixed point $\overline{z}^\star$ of $U\circ \sigma$ is globally asymptotically stable for \eqref{Exponential_Discount_Dynamic_2_Overall}.	Moreover, $x(t) = \sigma(z(t))$ converges to the logit equilibrium $\overline{x}^\star = \sigma(\overline{z}^\star)$ of the game $\mathcal{G}$.
	\end{prop}
	
	The above proposition essentially states that the convergence of the exponentially-discounted score dynamics \eqref{Exponential_Discount_Dynamic_2_Overall} relies on, individually, the Lipschitzness of the softmax function $\sigma$ and the game's payoff vector $U$. We illustrate this dependency using the following example. 
	
	\begin{example}
		By equivalence of norms, $\sigma$ is $\|\cdot\|_\infty$-contractive if $\sqrt{n}\lambda < 1$. Then for any game where the payoff vector $U$ is a $\|\cdot\|_\infty$-contraction, $U \circ \sigma$ is a $\|\cdot\|_\infty$-contraction. \autoref{prop:cominetti_norm_contraction} implies that the induced strategy $x(t) = \sigma(z(t))$ converges to the logit equilibrium $\overline{x}^\star \in \Updelta^{n-1}$.
	\end{example}
	
	\section{Conclusion and Open Problems}
	
	In this paper we have presented a thorough analysis of the softmax function using tools from convex analysis and monotone operator theory. We have shown that the softmax function is the monotone gradient map of the log-sum-exp function and that the inverse temperature parameter $\lambda$ determines the Lipschitz and co-coercivity properties of the softmax function. These properties allow for convenient constructions of convergence guarantees for score dynamics in general classes of games (see \cite{Gao}). We note that the structure of the reinforcement learning scheme is similar to those that arises in bandit and online learning (such as the Follow-the-Regularized-Leader (FTRL) and mirror descent algorithm \cite{Shalev}). We hope that researchers could adapt our results presented here and apply them to their domain-specific problems.

	Finally, for many applications in reinforcement learning, it is desirable to use a generalized version of the softmax function given by,
	 \begin{equation}
	 	\sigma_i(z) =  \frac{\exp(\lambda_i z_i)}{\textstyle\sum\limits_{j =1}^n \exp(\lambda_j z_j)}, 1 \leq i \leq n.
	 	\label{eqn:softmax_individual}
	 \end{equation}
	Here, each strategy $i$ is associated with an inverse temperature constant $\lambda_i > 0$, which can be adjusted independently to improve an agent's learning performance. The relationship between the individual parameters $\lambda_i$ with the convergence properties of score dynamics under the choice rule given by \eqref{eqn:softmax_individual} has been investigated in \cite{Cominetti} but is not yet fully characterized at this point.  It is of interest to extend the results presented in this paper for generalized versions of the softmax function \cite{Matrix_EXP} or adopt a monotone operator theoretic approach to analyze alternative forms of choice maps \cite{Asadi}.

\end{document}